\documentclass[10pt]{amsart}

\usepackage{amsfonts,amssymb,amscd,amstext}
\usepackage[a4paper,hmargin=3.5cm,vmargin=4cm]{geometry}
\usepackage{hyperref}
\usepackage{graphicx}
\usepackage[utf8]{inputenc}
\usepackage{esint}
\usepackage{enumerate}

\renewcommand{\leq}{\leqslant}
\renewcommand{\geq}{\geqslant}

\newcommand{\rr}{{\mathbb{R}}}

\newcommand{\hh}{{\mathbb{H}}}

\newcommand{\hhh}{\mathcal{H}}

\newcommand{\escpr}[1]{\big<#1\big>}
\newcommand{\Sg}{\Sigma}

\newcommand{\eps}{\varepsilon}

\newcommand{\ga}{\gamma}
\newcommand{\Ga}{\Gamma}

\newcommand{\mnh}{|N_{H}|}
\newcommand{\nt}{\escpr{N,T}}

\newcommand{\area}{A}

\newcommand{\n}{\nabla}

\newcommand{\ntnh}{\frac{\escpr{N,T}}{|N_h|}}
\newcommand{\ntnhtwo}{\bigg(\ntnh\bigg)^2}
\newcommand{\mink}{E(1,1)}

\newtheorem{theorem}{Theorem}[section]
\newtheorem{proposition}[theorem]{Proposition}
\newtheorem{lemma}[theorem]{Lemma}
\newtheorem{corollary}[theorem]{Corollary}

\theoremstyle{definition}

\newtheorem{problem}{Problem}

\theoremstyle{remark}

\newtheorem{remark}[theorem]{Remark}

\numberwithin{equation}{section}

\begin{document}

\title[Area-stationary and stable surfaces in Sol]{On the classification of complete area-stationary and stable surfaces in the sub-Riemannian Sol manifold}

\author[M.~Galli]{Matteo Galli} \address{Departamento de
Geometr\'{\i}a y Topolog\'{\i}a \\
Universidad de Granada \\ E--18071 Granada \\ Espa\~na}
\email{galli@ugr.es}

\date{\today}

\thanks{Research supported by  MCyT-Feder grant MTM2010-21206-C02-01 and J. A. grant P09-FMQ-5088}

\begin{abstract}
We study the classification of area-stationary and stable $C^2$ regular surfaces in the space of the rigid motions of the Minkowski plane $\mink$, equipped with its sub-Riemannian structure. We construct examples of area-stationary surfaces that are not foliated by sub-Riemannian geodesics. We also prove that there exist an infinite number of $C^2$ area-stationary surfaces with a singular curve. Finally we show the stability of $C^2$ area-stationary surfaces foliated by sub-Riemannian geodesics. 
\end{abstract}

\maketitle

\thispagestyle{empty}

\bibliographystyle{amsplain}

\tableofcontents

\section{Introduction}

The study of the sub-Riemannian area functional in three-dimensional pseudo-hermitian manifolds and in other sub-Riemannian spaces has been largely investigated in the last years, see \cite{Am-SC-Vi, MR2831583, BA-SC-Vi, MR2600502, MR2583494,  ChengHwang2nd,  CJHMY2,  CJHMY, CHY, DGNAV, DGNnotable,  Da-Ga-Nh-Pa, MR3044134, MR2979606, Ga-Nh,  Hl-Pa2, Ri-Ro-Hu,  Hu-Ro, Hu-Ro2, Ri1,  Ri-Ro, Ro,      Sh}, among others. 

One of the more interesting questions concerning the sub-Riemannian area functional is:

\begin{problem}\label{domanda1}
\emph{Which are the area-minimizing surfaces in a given three-dimensional contact sub-Riemannian manifold?}
\end{problem}

A surface $\Sg$ is \emph{area-minimizing} if $\area(\Sg)\leq \area(\tilde{\Sg})$, for any compact deformation $\tilde{\Sg}$ of $\Sg$. To answer the previous question, a natural preliminary step is the study of the \emph{area-stationary} surfaces, the critical points of the area functional.

\begin{problem}\label{domanda2}
\emph{Which are the area-stationary surfaces in a given three-dimensional contact sub-Riemannian manifold?}
\end{problem}

We will consider these questions in the class of $C^2$ regular surfaces. For a general introduction about the study of the area functional in sub-Riemannian spaces, we refer the interested reader to \cite{CDPT} and \cite{Gaphd}, that treat the case of $\hh^n$ and  the contact sub-Riemannian manifolds respectively.

In Sasakian space forms, the classification of $C^2$ area stationary surfaces was given in \cite{Ri-Ro-Hu} in the case of the Heisenberg group $\hh^n$ and in  \cite{Ro} for the Sasakian structures of  $S^3$ and $\widetilde{SL}_2(\rr)$.   In the case of pseudo-hermitian three-manifolds that are not Sasakian, the only known results concerning Problem \ref{domanda1}  and Problem \ref{domanda2} are given in \cite{MR3044134}, where the group of the rigid motions of the Euclidean plane $E(2)$ is studied. 

Concerning the three-dimensional pseudo-hermitian manifolds, we have the following classification result, \cite[Theorem~3.1]{Pe}, in terms of the Webster scalar curvature $W$ and of the pseudo-hermitian torsion $\tau$

\begin{proposition}\label{classification} Let  $M$ be a simply connected contact 3-manifold, homogeneous in the sense of Boothby and Wang, \cite{MR0112160}. Then $M$ is one of the following Lie group:
	\begin{itemize}
	\item [(1)] if $M$ is unimodular 
	         \begin{itemize}
	         \item the first Heisenberg group $\mathbb{H}^1$ when $W=|\tau|=0$;
	         \item the three-sphere group $SU(2)$ when $W> 2|\tau|$;
	         \item the group $\widetilde{SL(2,\mathbb{R})}$ when $-2|\tau|\neq W<2|\tau|$;
	         \item the group $\widetilde{E(2)}$, universal cover of the group of rigid motions of the Euclidean plane, when $W=2|\tau|>0$;
	         \item the group $E(1,1)$ of rigid motions of Minkowski 2-space, when $W=-2|\tau|<0$;
	         \end{itemize}
	\item [(2)] if $M$ is non-unimodular, the Lie algebra is given by
	        \[
	        [X,Y]=\alpha Y+2T ,\quad  [X,T]=\gamma Y, \quad  [Y,T]=0, \quad \alpha\neq 0,
	        \]
	        where $\{X,Y\}$ is an orthonormal basis of $\hhh$, $J(X)=Y$ and $T$ is the Reeb vector field. In this case $W<2|\tau|$ and when $\gamma=0$ the structure is Sasakian and $W=-\alpha^2$.
	\end{itemize}
\end{proposition}

About the models of the unimodular case, Problem \ref{domanda1} and Problem \ref{domanda2} are not investigated only for the case of the Sol geometry, modeling by the space $\mink$, and its study is the aim of this work. 

After some preliminaries, the paper is organized as follows. 

In Section \ref{sec:carcurves}, we compute explicitly the coordinates of the characteristic curves with given initial conditions. These curves play an important role in the study of area-stationary surfaces, since the regular part $\Sg-\Sg_0$ of a surface $\Sg$ is foliated by characteristic curves, that are not in general sub-Riemannian geodesics, since $\mink$ is characterized by a non-vanishing pseudo-hermitian torsion. 

Section \ref{sec:stationary} is the core of the paper. We first characterize the $C^2$ complete, area-stationary surfaces immersed in $\mink$ with singular points or singular curves that are sub-Riemannian geodesics. On the other hand, for the first time in the three-dimensional pseudo-hermitian setting, we also find examples of area-stationary surfaces that are not foliated by sub-Riemannian geodesics. We stress that these examples form an infinite family, i.e., given an horizontal curve $\Ga$, we can construct an area-stationary surface having $\Ga$ as singular set $\Sg_0$. 

Finally in Section \ref{sec:minimizing} we prove that complete area-stationary surfaces with non-empty singular set, whose characteristic curves are sub-Riemannian geodesics, are stable. We also find three families of non-singular planes that are area-minimizing, using a calibration argument. 

We remark that Section \ref{sec:minimizing} opens two interesting questions. Is a stable complete area-stationary  surface in $\mink$ with a singular curve  always foliated by sub-Riemannian geodesics in $\Sg-\Sg_0$? Do some  other complete stable area-stationary surfaces in $\mink$ with empty singular set exist?

\section{Preliminaries}
\label{sec:preliminaries}

\subsection{The group $E(1,1)$ of rigid motions of the Minkowski plane}
We consider the group of rigid motions of the Minkowski plane $E(1,1)$, that is a unimodular Lie group with a natural sub-Riemannian structure.  As a model of $E(1,1)$ we choose as underlying manifold $\rr^3$ with the following orthonormal basis of left-invariant vector fields
\begin{equation}\label{def:basis}
\begin{split}
X&=\frac{\partial}{\partial z}\\
Y&=\frac{1}{\sqrt{2}}\bigg(-e^{z}\frac{\partial}{\partial x}+e^{-z}\frac{\partial}{\partial y}\bigg)\\
T&=\frac{1}{\sqrt{2}}\bigg(e^{z}\frac{\partial}{\partial x}+e^{-z}\frac{\partial}{\partial y}\bigg).
\end{split}
\end{equation} 
We have that $\{X,Y\}$ is a orthonormal basis of the horizontal distribution $\hhh$ and $T$ is the Reeb vector field. The scalar product of two vector fields $W$ and $V$ with respect to the metric induced by the basis $\{X,Y,T\}$ will be often denoted by $\escpr{W,V}$. This structure of $E(1,1)$ is characterized by the following Lie brackets, \cite{MR0425012},
\begin{equation}\label{liebrakets}
\begin{split}
[X,Y]=-T \qquad [X,T]=-Y \qquad [Y,T]=0.
\end{split}
\end{equation} 
In fact, applying \cite[eq.~9.1 and eq.~9.3]{MR3044134} we obtain that the Webster scalar curvature is $W=-1/2$ and the matrix of the pseudo-hermitian torsion $\tau$ in the ${X,Y,T}$ basis is 
\[
\left( \begin{array}{ccc} 
0 & -\frac{1}{2}& 0 \\
-\frac{1}{2} & 0& 0\\
0 & 0 & 0 \end{array} \right).
\] 
The following derivatives can be easily computed
\begin{equation}\label{Gchistoffel}
\begin{split}
&\n_{X}X=0, \quad  \n_{Y}X=0,  \quad    \n_{T}X=\frac{1}{2}Y,    \\
&  \n_{X}Y=0, \quad    \n_{Y}Y=0, \quad    \n_{T}Y=-\frac{1}{2}X,
\end{split}
\end{equation}
where $\n$ denotes the pseudo-hermitian connection, \cite{Dr-To}.
Furthermore we have the characterization $-2|\tau|^2=W<0$ peculiar of $E(1,1)$, \cite{Pe}. We also define the involution $J$, the so-called complex structure, on $\hhh$ by $J(X)=Y$ and $J(Y)=-X$.

\subsection{The geometry of regular surfaces in $E(1,1)$}
We consider a $C^1$ surface $\Sg$ immersed in $E(1,1)$. We define the \emph{sub-Riemannian area} of $\Sg$ as
\[
\area(\Sg)=\int\limits_{\Sg}\mnh\, d\Sg,
\]
where $N_h$ denotes the projection of the Riemannian unit normal $N$ to $\hhh$ and $d\Sg$ denotes the Riemannian area element on $\Sg$. In the sequel we always denote by $N$ the inner unit normal. The singular set $\Sg_0$ is composed by the points in which $T\Sg$ coincides with $\hhh$. Outside $\Sg_0$, we can define the \emph{horizontal unit normal} 
\[
\nu_h:=\frac{N_h}{|N_h|}
\]
and the \emph{characteristic vector field} as $Z:=J(\nu_h)$. It is straightforward to verify that $\{Z,S\}$ is an orthonormal basis of $T\Sg$ outside $\Sg_0$, where
\[
S:=\nt \nu_h-\mnh T.
\]
Finally, outside $\Sg_0$, we define the \emph{mean curvature} of $\Sg$ by
\begin{equation}\label{def:H}
H:=-\escpr{\n_Z \nu_h,Z}.
\end{equation}
Given a surface $\Sg$ as zero level set of a function $u:\Omega\subset \mink\rightarrow \rr$, we can express
\begin{equation}\label{eq:nuhexpression}
\nu_h=-\frac{u_z X+\frac{1}{\sqrt{2}}(-e^z u_x+e^{-z}u_y)Y}{\sqrt{u_z^2+\frac{1}{2}(-e^z u_x+e^{-z}u_y)^2}}
\end{equation}
and
\begin{equation}\label{eq:Zexpression}
Z=\frac{\frac{1}{\sqrt{2}}(-e^z u_x+e^{-z}u_y) X-u_z Y}{\sqrt{u_z^2+\frac{1}{2}(-e^z u_x+e^{-z}u_y)^2}}. 
\end{equation}
We define a \emph{minimal surface} as a surface with vanishing mean curvature $H$. 

\begin{proposition} Let $\Sg$ be a minimal surface defined as the zero level set of a $C^2$ function $u:\Omega\subset\mink \rightarrow \rr$. Then $u$ satisfies the equation
\begin{equation}\label{eq:minimalsurface}
\begin{split}
u_{zz}(-e^z u_x+e^{-z}u_y)^2+u_z^2 (-e^{2z}u_{xx}-2u_{xy}+e^{-2z}u_{yy})\\
-u_z (-e^z u_x+e^{-z}u_y) (-2e^z u_{xz}-e^z u_x +2e^{-z}u_{yz}-e^{-z}u_y)=0
\end{split}
\end{equation}
 on $\Omega$.
\end{proposition}

\begin{proof} From \eqref{def:H}, \eqref{eq:nuhexpression} and \eqref{eq:Zexpression} we can find that $u$ has to satisfy
\begin{equation}\label{eq:minimalsurfacehalf}
Y(u)^2 X(X(u))-Y(u)X(u)Y(X(u))-Y(u)X(u)X(Y(u))+X(u)^2Y(Y(u))=0
\end{equation}
on $\Omega$. Now, using \eqref{def:basis}, we can transform \eqref{eq:minimalsurfacehalf} into \eqref{eq:minimalsurface}.
\end{proof}
We will call \eqref{eq:minimalsurface} the \emph{minimal surface equation}. 

\begin{remark}\label{immediateminimaremark} From \eqref{eq:minimalsurfacehalf}, it  is immediate to note that a surface $\Sg$ satisfying $u_z\equiv 0$ or $-e^z u_x+e^{-z}u_y\equiv 0$ is always minimal. 
\end{remark}

In the following Lemma, we compute some important quantities related to the torsion and the geometry of a surface. It follows from \cite[eq.~9.8]{MR3044134},

\begin{lemma}\label{lem:calcolotausullasuperficie}
Let $\Sg$ be a $C^1$ surface in $\mink$, then we have
\begin{equation*}
\begin{split}
\escpr{\tau(Z),Z}&=-\escpr{Z,X}\escpr{Z,Y}=\escpr{\nu_h,X}\escpr{\nu_h,Y}=-\escpr{\tau(\nu_h),\nu_h},\\
\escpr{\tau(Z),\nu_h}&=\frac{1}{2}(\escpr{Z,Y}^2- \escpr{Z,X}^2).
\end{split}
\end{equation*}
\end{lemma}

\section{Characteristic curves in $\mink$}\label{sec:carcurves}

In this section we will study the equation of the integral curves of $Z$ on $\Sg$, that are known as \emph{characteristic curves}. It is well-known that a surface with constant mean curvature $H$ is foliated by characteristic curves in $\Sg-\Sg_0$. In general, a \emph{characteristic curve} is an arc-length parametrized horizontal curve $\ga$ in $\mink$, that satisfies the equation 
\begin{equation}\label{eq:carcurves}
\n_{\dot{\ga}} \dot{\ga}+HJ(\dot{\ga})=0,
\end{equation}
where $\dot{\ga}$ denotes the tangent vector along $\ga$ and $H$ is the (constant) curvature of $\ga$. 
We stress that a curve $\ga$ satisfying  \eqref{eq:carcurves} is not a sub-Riemannian geodesic. In fact a characteristic curve $\ga$ is a sub-Riemannian geodesic if and only if $H=0$ and $\dot{\ga}$ satisfies the additional equation
\begin{equation}
       \escpr{\tau(\dot{\ga}),\dot{\ga}}=0,
\end{equation}
see \cite[Proposition~15]{Ru}, that forces $\ga$ to be an integral curve of $X$ or $Y$ by Lemma \ref{lem:calcolotausullasuperficie}.

\begin{proposition}\label{lemma:carcurves} Let $\ga$ be a characteristic curve in $\mink$ with curvature $H=0$. Then $\ga$ belongs to the family of curves
\begin{equation}
\ga(t)=(x_0+\dot{x}_0 t, y_0+\dot{y}_0 t, z_0)
\end{equation}
or to the family
\begin{equation}
\ga(t)=\bigg(x_0+\frac{\dot{x}_0}{\dot{z}_0}(e^{\dot{z}_0t}-1), y_0-\frac{\dot{y}_0}{\dot{z}_0}(e^{-\dot{z}_0t}-1),
z_0+\dot{z}_0t\bigg),
\end{equation}
where $\ga(0)=(x_0,y_0,z_0)$ and $\dot{\ga}(0)=(\dot{x}_0,\dot{y}_0,\dot{z}_0)$.
\end{proposition}

\begin{proof} We consider the curve $\ga:I\rightarrow \Sg$, where $I$ denotes an interval. We express $\ga(t)=(x(t),y(t),z(t))$ and we get
\begin{equation}\label{eq:gammadot}
\begin{split}
\dot{\ga}(t)&=\dot{x} \frac{\partial}{\partial x}+\dot{y} \frac{\partial}{\partial y}+\dot{z} \frac{\partial}{\partial z}\\
&=\dot{z} X+\frac{1}{\sqrt{2}}(\dot{y}e^z-\dot{x}e^{-z})Y+\frac{1}{\sqrt{2}}(\dot{y}e^z+\dot{x}e^{-z})T,
\end{split}
\end{equation}
since 
\begin{equation*}
\begin{split}
\frac{\partial}{\partial x}&=\frac{1}{\sqrt{2}}e^{-z}(T-Y)\\
\frac{\partial}{\partial y}&=\frac{1}{\sqrt{2}}e^{z}(Y+T).
\end{split}
\end{equation*}
From \eqref{eq:gammadot} and the fact that $\ga$ is horizontal, we have
\begin{equation}\label{eq:horcondgamma}
\dot{y}e^z+\dot{x}e^{-z}=0.
\end{equation}
Now $\n_{\dot{\ga}}\dot{\ga}=0$ is equivalent to the system
\begin{equation}\label{eq:nablagaga}
 \begin{cases}
\dot{z}=\dot{z}_0 \\
\dot{y}e^z-\dot{x}e^{-z}=c_0
\end{cases},
\end{equation}
where $\dot{z}_0$ and $c_0$ are constants. We distinguish two cases. The first one corresponds to $\dot{z}_0=0$. This means that $z=z_0$, with $z_0\in\rr$, and so \eqref{eq:horcondgamma} and \eqref{eq:nablagaga} are reduced to
\begin{equation}\label{sistemadaintegrare}
\begin{cases}
2\dot{y}=e^{-z_0}c_0 \\
2\dot{x}=-e^{z_0}c_0
\end{cases},
\end{equation}
that implies $\ga(t)=(x_0-e^{z_0}(c_0/2) t, y_0+e^{-z_0}(c_0/2) t, z_0)$, where $c_0\neq 0$ and $x_0,y_0\in\rr$. 

The second possibility is $\dot{z_0}\neq 0$, that implies $z(t)=z_0+\dot{z}_0 t$, with $z_0\in\rr$. In this case integrating \eqref{sistemadaintegrare} we obtain
\[
\ga(t)=(x_0+\frac{c_0 e^{z_0}}{2\dot{z}_0}-\frac{c_0}{2\dot{z}_0}e^{z_0+\dot{z}_0t}, y_0+\frac{c_0 e^{-z_0}}{2\dot{z}_0}-\frac{c_0}{2\dot{z}_0}e^{-(z_0+\dot{z}_0)t},
z_0+\dot{z}_0t),
\]
where $\ga(0)=(x_0,y_0,z_0)$. Finally, to conclude the result, we note that 
\[
\frac{c_0}{2}=\dot{y}_0e^{z_0}=-\dot{x}_0e^{-z_0}.
\]
\end{proof}

\section{Complete area-stationary surfaces with non-empty singular set in $\mink$}\label{sec:stationary}


\subsection{Complete area-stationary surfaces containing isolated singular points}  

The local structure of a $C^1$ surface $\Sg$ with prescribed mean curvature $H\in C$, in a neighborhood of an isolated singular point, is well understood, \cite[Theorem~D and Corollary~E]{CJHMY2}. In the less general case of a bounded mean curvature surface of class $C^2$, applying  \cite[Theorem~B and Section~7]{CJHMY}, we have

\begin{lemma}\label{singularsetthpoint} Let $\Sg$ be a ${C}^2$ oriented immersed surface with constant mean curvature $H$ in $\mink$. If $p\in\Sg_{0}$ is an isolated singular point,  then, there exists $r > 0$ and $\lambda\in 
\mathbb{R}$ such that the set described as 
\begin{equation*}
D_{r} (p) = \{\gamma^{H}_{p,v}  (s)| v \in T_{p}\Sg, |v| = 1, s \in [0, r)\}, 
\end{equation*}
is an open neighborhood of $p$ in $\Sg$, where $\gamma^{H}_{p,v}$ denote the characteristic curve starting from $p$ in the direction $v$ with curvature $H$.
\end{lemma}
 
First we construct the unique example, up to contact isometries, of a minimal surface with isolated singular points.  

\begin{proposition}\label{prop:minimalstationarysingularpoint} Let $\Sg$ be a $C^2$ complete, area-stationary surface immersed in $\mink$ with $H=0$ and with an isolated singular point $p_0=(x_0,y_0,z_0)$. Then $\Sg=\{(x,y,z)\in\mink : e^{z-z_0}(y-y_0)+x-x_0=0\}$.
\end{proposition}

\begin{proof} By Lemma \ref{singularsetthpoint}, the only possible way to construct a complete area-stationary surface, with a singular point $p_0$, is consider the union of all characteristic curves $\ga$ of curvature $0$ with initial conditions $\ga(0)=p_0$ and $\dot{\ga}(0)\in T_{p_0}\Sg=\hhh_{p_0}$, $|\dot{\ga}(0)|=1$. We can suppose $p_0=0$, since $\mink$ is homogeneous. 

We consider the initial velocities
\begin{equation*}
\begin{split}
\dot{\ga}_\alpha(0)&=\cos(\alpha)X(0)+\sin(\alpha)Y(0)\\
&=\cos(\alpha)\frac{\partial}{\partial z}(0)+\frac{\sin(\alpha)}{\sqrt{2}}\bigg(-\frac{\partial}{\partial x}(0)+\frac{\partial}{\partial y}(0)\bigg),
\end{split}
\end{equation*}
for $\alpha\in[0,2\pi[$. In this way we obtain as characteristic curves
\begin{equation}
\ga_\alpha(t)=\bigg(-\frac{\sin(\alpha)}{\sqrt{2}\cos(\alpha)}(e^{\cos(\alpha)t}-1), -\frac{\sin(\alpha)}{\sqrt{2}\cos(\alpha)}(e^{-\cos(\alpha)t}-1),
\cos(\alpha)t\bigg),
\end{equation}
for $\alpha\in]0,2\pi[$ and $\ga_0(t)=(0,0,t)$ for $\alpha=0$. At this point it is easy show that $\Sg$ is the zero level set of the function $e^zy+x=0$ (or equivalently $e^{-z}x+y=0$), that satisfies \eqref{eq:minimalsurface}.

\end{proof}

%

\subsection{Complete area-stationary surfaces containing singular curves}
From \cite[Corollary~5.4]{MR3044134} we have 

\begin{lemma}\label{charmeetsingular} Let $\Sg$ be a ${C}^2$  minimal surface with non-empty singular set $\Sg_{0}$ immersed in $\mink$. Then $\Sg$ is area stationary if and only if the characteristic curves meet the singular curves orthogonally with respect the metric $\escpr{\, ,\, }$, induced by the orthonormal basis \eqref{def:basis}. 
\end{lemma}

In the following lemma, we prove that a minimal area-stationary surface can not contain more than a singular curve.

\begin{lemma} Let $\Sg$ be a $C^2$ complete, minimal, area-stationary surface, containing a singular curve $\Ga$, immersed in $\mink$. Then $\Sg$ cannot contain more singular curves. 
\end{lemma}

\begin{proof} We consider a singular curve $\Ga(\eps)=(x(\eps),y(\eps),z(\eps))$ in $\Sg$. Then, as $\Sg$ is foliated by characteristic curves, we can parametrize it by the map 
\[
F(\eps,t)=\ga_\eps(t)=(x(\eps,t),y(\eps,t),z(\eps,t)),
\]
where $\ga_\eps(t)$ is the characteristic curves with initial data $\ga_\eps(0)=\Ga(\eps)$ and 
\begin{equation}
\begin{split}
\dot{\ga}_\eps(0)&=J(\dot{\Ga}(\eps))=\dot{z}(\eps)J(X)+\frac{1}{\sqrt{2}}(\dot{y}(\eps)e^{z(\eps})-\dot{x}(\eps)e^{-z(\eps)})J(Y)\\
&=  \frac{1}{\sqrt{2}}(-\dot{z}(\eps)e^{z(\eps)},\dot{z}(\eps)e^{-z(\eps)},\dot{x}(\eps)e^{-z(\eps)}-\dot{y}(\eps)e^{z(\eps)}).
\end{split}
\end{equation}
We define the function $V_\eps(t):=(\partial F/\partial \eps)(t,\eps)$ that is a smooth Jacobi-like vector field along $\ga_\eps(t)$, \cite[Section~4]{MR3044134}. We have that, in a singular point $(\eps,t)$, the vertical component of $V_\eps$ 
\[
\escpr{V_\eps,T}(\eps,t)=\frac{\partial x}{\partial\eps}(\eps,t)e^{-z(\eps,t)}+\frac{\partial y}{\partial\eps}(\eps,t)e^{z(\eps,t)}
\]
vanishes. We suppose that $\Ga$ is not an integral curve of $X$ or $Y$. Then, from the following expression of the component of $F(\eps,t)$
\begin{equation}\label{eq:parametrizzazionecurvaorizzontale}
\begin{split}
x(\eps,t)&=x(\eps)+\frac{\dot{z}(\eps)e^{z(\eps)}}{\dot{x}(\eps)e^{-z(\eps)}-\dot{y}(\eps)e^{z(\eps)}}(e^{(\dot{x}(\eps)e^{-z(\eps)}-\dot{y}(\eps)e^{z(\eps)})t/\sqrt{2}}-1)\\
y(\eps,t)&=y(\eps)-\frac{\dot{z}(\eps)e^{-z(\eps)}}{\dot{x}(\eps)e^{-z(\eps)}-\dot{y}(\eps)e^{z(\eps)}}(e^{-(\dot{x}(\eps)e^{-z(\eps)}-\dot{y}(\eps)e^{z(\eps)})t/\sqrt{2}}-1)\\
z(\eps,t)&=z(\eps)+\frac{\dot{x}(\eps)e^{-z(\eps)}-\dot{y}(\eps)e^{z(\eps)}}{\sqrt{2}}t,
\end{split}
\end{equation}
we have 
\begin{equation}
\begin{split}
\escpr{V_\eps,T}(\eps,t)&=\bigg \{\dot{x}(\eps)e^{-z(\eps)}+\frac{\ddot{z}(\eps)}{\dot{x}(\eps)e^{-z(\eps)}-\dot{y}(\eps)e^{z(\eps)}}-            
\frac{\dot{z}(\eps)\frac{\partial}{\partial\eps}(\dot{x}(\eps)e^{-z(\eps)}-\dot{y}(\eps)e^{z(\eps)})}{(\dot{x}(\eps)e^{-z(\eps)}-\dot{y}(\eps)e^{z(\eps)})^2}
\bigg\}\cdot\\
&\hspace{2cm}\cdot(e^{-(\dot{x}(\eps)e^{-z(\eps)}-\dot{y}(\eps)e^{z(\eps)})t/\sqrt{2}}-e^{(\dot{x}(\eps)e^{-z(\eps)}-\dot{y}(\eps)e^{z(\eps)})t/\sqrt{2}})\\
&+\frac{\dot{z}(\eps)^2}{(\dot{x}(\eps)e^{-z(\eps)}-\dot{y}(\eps)e^{z(\eps)})^2}(e^{-(\dot{x}(\eps)e^{-z(\eps)}-\dot{y}(\eps)e^{z(\eps)})t/\sqrt{2}}+e^{(\dot{x}(\eps)e^{-z(\eps)}-\dot{y}(\eps)e^{z(\eps)})t/\sqrt{2}}-2),
\end{split}
\end{equation}
that vanishes only for the values $(\eps,0)$, for positive values of $t$. On the other hand, if $\Ga$ is an integral curve of $Y$, we get 
\begin{equation}
\begin{split}
x(\eps,t)&=x(\eps)\\
y(\eps,t)&=y(\eps)\\
z(\eps,t)&=z(\eps)+\frac{\dot{x}(\eps)e^{-z(\eps)}-\dot{y}(\eps)e^{z(\eps)}}{\sqrt{2}}t;
\end{split}
\end{equation}
if $\Ga$ is an integral curve of $X$, we have
\begin{equation}
\begin{split}
x(\eps,t)&=x(\eps)-\frac{\dot{z}(\eps)e^{z(\eps)}}{\sqrt{2}}t\\
y(\eps,t)&=y(\eps)+\frac{\dot{z}(\eps)e^{-z(\eps)}}{\sqrt{2}}t\\
z(\eps,t)&=z(\eps).
\end{split}
\end{equation}
In both cases, the singular set is only the curve $\Ga(\eps)$.

\end{proof}


The vertical component of $V_\eps$ can be computed more directly using \cite[Proposition~4.3]{MR3044134}, since $H=0$. On the other hand, the explicit computation of the components of the parametrization $F(\eps,t)$ allows us to characterize all the $C^2$ area-stationary complete surfaces with a singular curve that is a characteristic curve of curvature $0$. We stress that, when the characteristic curves are sub-Riemannian geodesics, these examples can also be constructed from Remark \ref{immediateminimaremark}.

\begin{proposition} Let $\Sg$ be an area-stationary surface with $H=0$, with a singular curve $\Ga$ that is a characteristic curve of curvature $0$. Then, if $\Ga$ is a sub-Riemannian geodesic, $\Sg$ belongs to one of the families
\begin{itemize}
\item[(i)] $\{ a x+b y+ c=0: (x,y,z)\in\mink, a,b,c\in \rr\}$;
\item[(ii)] $\{e^{z-z_0} (y-y_0)+e^{z_0-z}(x-x_0)=0: (x,y,z)\in \mink, x_0,y_0,z_0\in \rr\}$.
\end{itemize}
Otherwise, we suppose that $\Ga$ is a characteristic curve passing through $(x_0,y_0,z_0)$ with velocity $(\dot{x}_0,\dot{y}_0,\dot{z}_0)$, $\dot{x}_0,\dot{y}_0,\dot{z}_0\neq 0$. We can parametrize $\Sg$ by $F:\rr^2\rightarrow \mink$, with $F(\eps,t)=(x(\eps,t),y(\eps,t),z(\eps,t))$ and 
\begin{equation}\label{eq:parametrizationsingularcurvenotgeo}
\begin{split}
x(\eps,t)&=x_0+\frac{\dot{x}_0}{\dot{z}_0}(e^{\dot{z}_0\eps}-1)+\frac{\dot{z}_0e^{z_0+\dot{z}_0\eps}}{\dot{x}_0e^{-z_0}-\dot{y}_0e^{z_0}}(e^{(\dot{x}_0e^{-z_0}-\dot{y}_0e^{z_0})t/\sqrt{2}}-1)\\
y(\eps,t)&=y_0-\frac{\dot{y}_0}{\dot{z}_0}(e^{-\dot{z}_0\eps}-1)-\frac{\dot{z}_0e^{-z_0-\dot{z}\eps}}{\dot{x}_0e^{-z_0}-\dot{y}_0e^{z_0}}(e^{-(\dot{x}_0e^{-z_0}-\dot{y}_0e^{z_0})t/\sqrt{2}}-1)\\
z(\eps,t)&=z_0+\dot{z}_0\eps+\frac{\dot{x}_0e^{-z_0}-\dot{y}_0e^{z_0}}{\sqrt{2}}t.
\end{split}
\end{equation}

\end{proposition}

\begin{remark} We note that the surfaces parametrized by \eqref{eq:parametrizationsingularcurvenotgeo} are the first examples of area-stationary surfaces  that are not foliated by sub-Riemannian geodesics in three-dimensional contact sub-Riemannian manifolds, up to our knowledge. In fact this phenomena do not appear in the roto-translation group, \cite[Lemma~10.4]{MR3044134}, even if its pseudo-hermitian torsion is non-vanishing.  In that case, the presence of two singular curves force the the surface to be foliated by sub-Riemannian geodesics or to be not area-stationary. On the other hand, it is well-known that a minimal surface is foliated by sub-Riemannian geodesics in any three-dimensional Sasakian manifold.

\end{remark}

\begin{remark} Given any horizontal curve $\Ga=(x(\eps),y(\eps),z(\eps))$ in $\mink$, we stress that \eqref{eq:parametrizzazionecurvaorizzontale} provide a parametrization $F(\eps,t):\rr^2\rightarrow \Sg \subset \mink$ of a complete area-stationary surface $\Sg$ with $\Sg_0=\Ga$.

\end{remark}

\section{Complete area-minimizing surfaces in $\mink$}\label{sec:minimizing}

\subsection{Complete area-minimizing surfaces with empty singular set} 

In \cite[Proposition~9.8]{MR3044134} is shown a general necessary condition for the stability of a non-singular surface in pseudo-hermitian Lie groups. This condition states that the quantity
\[
W-\escpr{\tau(Z),\nu_h}=\escpr{\nu_h,Y}^2-1=\escpr{Z,X}^2-1
\]
has to be always non-positive.  This condition is trivial in $\mink$ due to the negativity of the Webster scalar curvature. On the other hand it has been used strongly in the classification of the stable, area-stationary surfaces without singular points in the manifolds $\hh^1$, $SU(2)$ and $\widetilde{E(2)}$, see \cite{MR3044134,Ri-Ro-Hu, Ro}. In any way, we can prove the following

\begin{proposition}\label{lemmapianiminimizzanti} The families of planes
\begin{itemize}
\item[(i)] $\{ x+c=0: (x,y,z)\in\mink, c\in\rr\}$;
\item[(ii)] $\{ y+c=0: (x,y,z)\in\mink, c\in\rr\}$;
\item[(iii)] $\{ z+c=0: (x,y,z)\in\mink, c\in\rr\}$;
\end{itemize}
are area-stationary, foliated by sub-Riemannian geodesics, and area-minimizing. 
\end{proposition}

\begin{proof}
We prove the result for $\Sg=\{ x=0: (x,y,z)\in\mink\}$, since all the cases are similar. In this case, from \eqref{eq:nuhexpression} and \eqref{eq:Zexpression}, we have
\[
\nu_h=Y \qquad Z=-X.
\]
So the integral curves of $Z$ are sub-Riemannian geodesics and $\Sg_0=\emptyset$. Now Remark \ref{immediateminimaremark} implies that $\Sg$ is area-stationary. Finally we can foliate a neighborhood of $\Sg$ in $\mink$ simply translating $\Sg$. We obtain a foliation by area-stationary surfaces and a standard calibration argument imply that $\Sg$ is area-minimizing, see for example \cite{BA-SC-Vi}, \cite{Ri1} or \cite[\S~5]{Ri-Ro}.
\end{proof}

\begin{remark}
The family of planes $\{ ax+by+cz+d=0: (x,y,z)\in\mink, a,b,c,d \in\rr\}$ are not minimal, since they do not satisfy \eqref{eq:minimalsurface}.
\end{remark}

A very natural question is: are the planes in Proposition \ref{lemmapianiminimizzanti} the unique complete area-minimizing surfaces with empty singular set in $\mink$? 

We have only been able to find the following sufficient condition

\begin{lemma}\label{lem:sufficientstabilityempty}  Let $\Sg$ be a $C^2$ complete oriented minimal surface immersed in \mink, with empty singular set $\Sg_0$. If on $\Sg$ there holds $\escpr{N,T}\leq 0$, then $\Sg$ is stable.
\end{lemma}

\begin{proof} Taking into account the expression for stability operator for non-singular surfaces in \cite[Lemma~8.3]{MR3044134}, we only need to show that 
\[
2Z\bigg(\ntnh \bigg)+\ntnhtwo\leq 0
\]
on $\Sg$. Given a point $p$ in $\Sg$, let $I$ an open interval containing the origin and $\alpha:I\rightarrow \Sg$ a piece of the integral curve of $S$ passing through $p$. Consider the characteristic curve $\gamma_\eps(s)$ of $\Sg$ with $\gamma_\eps(0)=\alpha(\eps)$. We define the map $F:I\times\rr\rightarrow \Sg$ given by $F(\eps,s)=\gamma_\eps(s)$ and denote $V(s):=(\partial F/\partial \eps)(0,s)$ which is a Jacobi-like vector field along $\gamma_0$, see \cite[Proposition~4.3]{MR3044134}. Denoting by $'$ the derivatives of functions depending on $s$, and the covariant derivative along $\gamma_0$ respect to $\n$ and $\dot{\gamma}_0$ by $Z$. Using \cite[Lemma~3.1, Eq.~4.4 and Eq.~4.5]{MR3044134} we get
\begin{equation}\label{vt0s}
\escpr{V,T}(0)=-\mnh,
\end{equation}
\begin{equation}\label{vt1s}
\escpr{V,T}'(0)=-\nt,
\end{equation}
\begin{equation}\label{vt2s}
\escpr{V,T}''(0)=-\mnh \bigg( Z\bigg(\ntnh \bigg)+\ntnhtwo   \bigg)
\end{equation}
It is easy to show that $g(V,T)$ never vanishes along $\gamma_0$ as $\Sg_0$ is empty, see \cite[Proof of Lemma~9.5]{MR3044134}. On the other hand, by \cite[Proposition~4.3]{MR3044134} and Lemma \ref{lem:calcolotausullasuperficie}, we have that $\escpr{V,T}$ satisfies the ordinary differential equation
\[
\escpr{V,T}'''(s)-\escpr{Z,X}^2\escpr{V,T}'(s)=0
\]
along $\ga_0$. We suppose that $\escpr{Z,X}\neq 0$. Taking into account the initial conditions \eqref{vt0s}, \eqref{vt1s} and \eqref{vt2s}, we obtain
\[
\escpr{V,T}(s)=a \cosh(|\escpr{Z,X}|s)+ b \sinh(|\escpr{Z,X}|s)+c,
\] 
where 
\[
a=\frac{\mnh \bigg( Z\bigg(\ntnh \bigg)+\ntnhtwo   \bigg)}{\escpr{X,Z}^2}, \quad b= -\frac{\nt}{|\escpr{Z,X}|}
\]
and
\[ 
c= -\mnh -\frac{\mnh \bigg( Z\bigg(\ntnh \bigg)+\ntnhtwo   \bigg)}{\escpr{X,Z}^2}.
\]
We have that $\escpr{V,T}(s)\neq 0$ implies 
\[
a+b=\frac{\mnh \bigg( Z\bigg(\ntnh \bigg)+\ntnhtwo   \bigg)}{\escpr{X,Z}^2} -\frac{\nt}{|\escpr{Z,X}|}\leq 0.
\] 
Then we can conclude 
\begin{equation*}
\begin{split}
2&Z\bigg(\ntnh \bigg)+\ntnhtwo\leq2 \{Z\bigg(\ntnh \bigg)+\ntnhtwo\}\leq 2 |\escpr{Z,X}|\ntnh\leq 0
\end{split}
\end{equation*}
on $\ga_0$. Now since the choice of $p$ is arbitrary, we get the statement. 

If $\escpr{Z,X}=0$, we conclude that $\Sg$ is stable if and only if $\nt=0$, by \cite[Proposition~9.8]{MR3044134}.

\end{proof}

\begin{remark} We note that the surfaces described in the points (i), (ii), (iii) of Proposition \ref{lemmapianiminimizzanti}, are characterized by $\nt=-e^z/\sqrt{2}$, $\nt=-e^z/\sqrt{2}$ and $\nt\equiv 0$ respectively, where $N$ denotes the inward unit normal on $\Sg$. In the third family the planes are vertical surfaces and they satisfy $W-\escpr{\tau(Z),\nu_h}\equiv 0$. 
\end{remark}

Taking into account the geometric invariants of $\mink$, we expect the existence of other examples of complete oriented minimal surface with empty singular set.

\subsection{Complete area-minimizing surfaces with non-empty singular set} 

We consider the stability operator constructed in \cite[Theorem~8.6]{MR3044134}

\begin{lemma}\label{Stability operator II} Let $\Sg$ be a $C^2$ oriented minimal surface immersed in \mink, with singular set $\Sg_0$ and $\partial\Sg=\emptyset$. If $\Sg$ is stable then, for any function $u\in C^1_0(\Sg)$ such that $Z(u)=0$ in a tubular neighborhood of a singular curve and constant in a tubular neighborhood of an isolated singular point, we have $Q(u)\geq 0$, where
\begin{equation*}
\begin{split}
Q(u):=&\int\limits_{\Sg}\{|N_h|^{-1}Z(u)^2+\mnh( (1+\escpr{Z,Y}^2)-(\mnh (1/2-\escpr{Z,Y}^2)-\escpr{\n_S \nu_h,Z}   )^2  )u^2\}d\Sg\\
&+4\int\limits_{(\Sg_0)_c}\nt\escpr{Z,Y}^2\escpr{Z,\nu}u^2d(\Sg_0)_c+\int\limits_{(\Sg_0)_c} S(u)^2  d(\Sg_0)_c.
\end{split}
\end{equation*}
Here $d(\Sg_0)_c$ is the Riemannian length measure on $(\Sg_0)_c$ and $\nu$ is the external unit normal to $(\Sg_0)_c$.
\end{lemma}

\begin{corollary}
Let $\Sg$ be a plane in the family $\{ a x+b y+ c=0: (x,y,z)\in\mink, a,b,c\in \rr\}$. Then $\Sg$ is stable.
\end{corollary}

\begin{proof} We know that $\Sg$ is area-stationary with a singular line, obtained intersecting $\Sg$ with the plane $z=\log\sqrt{b/a}$.  From \eqref{eq:Zexpression} we get
\[
Z=\frac{-b e^z+ a e^{-z}}{|-b e^z+ a e^{-z}|}X, 
\]
which is orthogonal to the singular line. Since 
\[
\escpr{\n_S \nu_h,Z}=\escpr{\n_S Y,X}=\frac{\mnh}{2},
\]
we have that the stability operator
\[
Q(u)=\int\limits_{\Sg}\{|N_h|^{-1}Z(u)^2+\mnh\nt^2 u^2\}d\Sg+\int\limits_{\Sg_0} S(u)^2  d\Sg_0
\]
is always non-negative for any admisible test function $u$.
\end{proof}

\begin{remark} The planes $\{ a x+b y+ c=0: (x,y,z)\in\mink, a,b,c\in \rr\}$ are also area-minimzing, by calibration arguments. 

\end{remark}

\begin{corollary} We consider the surface $\Sg=\{e^z y+e^{-z}x=0: (x,y,z)\in \mink\}$. Then $\Sg$ is stable. 

\end{corollary}

\begin{proof} From \eqref{eq:Zexpression} we get
\[
Z=-\frac{(e^z y-e^{-z}x)Y}{|e^z y-e^{-z}x|}
\]
and $\Sg_0=\{(0,0,z):(x,y,z)\in \mink\}$. From \eqref{Gchistoffel} we have
\[
\escpr{\n_S \nu_h,Z}=\escpr{\n_S Y,X}=-\frac{\mnh}{2},
\]
which implies
\[
Q(u)=\int\limits_{\Sg}\{|N_h|^{-1}Z(u)^2+2\mnh^2 u^2\}d\Sg+\int\limits_{\Sg_0} S(u)^2  d\Sg_0+ 4\int\limits_{\Sg_0} u^2  d\Sg_0\geq 0,
\]
for all admissible $u$.
 \end{proof}

 \begin{corollary} The surfaces defined in Proposition \ref{prop:minimalstationarysingularpoint} are stable. 
\end{corollary}

\begin{proof} For simplicity we will prove the statement in the case of $x_0=y_0=z_0=0$. We note that, since $\Sg_0=(0,0,0)$, the argument in the proof of Lemma \ref{lem:sufficientstabilityempty} works and  the condition $\escpr{N,T}=-(1+e^z)/\sqrt{2}\leq 0$ is a sufficient condition for the stability in the complementary of any tubular neighborhood  of  $\Sg_0$. Finally we observe that the stability operator in Lemma \ref{Stability operator II} does not give contributions of the singular set in the case of isolated singular points. 
 \end{proof}

\def\cprime{$'$}
\providecommand{\bysame}{\leavevmode\hbox to3em{\hrulefill}\thinspace}
\providecommand{\MR}{\relax\ifhmode\unskip\space\fi MR }
\providecommand{\MRhref}[2]{%
  \href{http://www.ams.org/mathscinet-getitem?mr=#1}{#2}
}
\providecommand{\href}[2]{#2}

\end{document}